\documentclass{article}
\usepackage{amssymb,amscd,graphpap}
\usepackage{comment}
\usepackage[fleqn]{amsmath}
\usepackage[left=4cm,right=4cm,top=4cm,bottom=4cm]{geometry}
\usepackage{graphicx}
\usepackage{multicol}
\usepackage{changepage}
\usepackage[utf8]{inputenc}
\usepackage{float}
\usepackage{datetime}
\usepackage{amsthm}
\usepackage[mathscr]{euscript}
 \let\mathscr\relax

\usepackage[scr]{rsfso}

\usepackage{rotating}
\usepackage{thmtools, thm-restate}
\usepackage{pgffor}
\usepackage{tabto}
\usepackage{authblk}

\usepackage{tikz}
\usetikzlibrary{automata}
\usetikzlibrary{arrows} 
\usepgflibrary{arrows} 

\definecolor{awesome}{rgb}{1.0, 0.13, 0.32}
\definecolor{iceberg}{rgb}{0.44, 0.65, 0.82}
\definecolor{forest}{rgb}{0.0, 0.27, 0.13}
\definecolor{gold}{rgb}{0.99, 0.76, 0.0}
\definecolor{teal}{rgb}{0.0, 0.5, 0.5}

\newtheorem{theorem}{Theorem}[section]
\newtheorem{lemma}[theorem]{Lemma}

\newtheorem{conjecture}[theorem]{Conjecture}
\newtheorem{note}[theorem]{Note}
\newtheorem{definition}{Definition}[section]

\date{January $23^{rd}$ 2021}

\title{Collatz convergence is a Hydra Game}

\author{Alexander Rahn\thanks{Nuremberg Institute of Technology, Keßlerpl. 12 90489 Nuremberg, Germany}, \hspace{1pt}   Eldar Sultanow\thanks{Potsdam University, Chair of Business Informatics, Processes and Systems, Karl-Marx Straße 67, 14482, Potsdam, Germany
\\Capgemini, Bahnhofstraße 30, 90402, Nuremberg, Germany} \hspace{1pt} and Idriss J. Aberkane\thanks{Unesco-Unitwin Complex Systems Digital Campus, Chair of Prof. Pierre Collet, Pierre Collet, ICUBE - UMR CNRS 7357, 4 rue Kirschleger, 67000 Strasbourg, France. Corresponding author: idriss.aberkane@polytechnique.edu}}

\begin{document}

\maketitle

\begin{abstract}
The Collatz dynamic is known to generate a complex quiver of sequences over natural numbers which inflation propensity remains so unpredictable it could be used to generate reliable proof of work algorithms for the cryptocurrency industry. Here we establish an \textit{ad hoc} equivalent of modular arithmetic for Collatz sequences to automatically demonstrate the convergence of infinite quivers of numbers, based on five arithmetic rules we prove apply on the entire Collatz dynamic and which we further simulate to gain insight on their graph geometry and computational properties. We then formally demonstrate these rules define an automaton that is playing a Hydra game on the graph of undecided numbers we also prove is embedded in $24\mathbb{N}-7$, proving that in ZFC the Collatz conjecture is true, before giving a promising direction to also prove it in Peano arithmetic. 
\end{abstract}

\section{Introduction}
\label{sec:introduction}
The dynamical system generated by the $3n+1$ problem is known to create complex quivers over $\mathbb{N}$, one of the most picturesque being the so-called "Collatz Feather" or "Collatz Seaweed", a name popularized by Clojure programmer Oliver Caldwell in 2017. The inflation propensity of Collatz orbits remains so unpredictable it can form the core of a reliable proof-of-work algorithm for Blockchain solutions \cite{Ref_Bocart_2018}, with groundbreaking applications to the field of number-theoretical cryptography as such algorithms are unrelated to primes yet, being based on the class of congruential graphs, still allow for a wide diversity of practical variants. If Bocart thus demonstrated that graph-theoretical approaches to the $3n+1$ problem can be very fertile to applied mathematics, the authors have also endeavored to demonstrate its pure number-theoretical interest prior to this work \cite{Ref_Aberkane_2017}, \cite{Ref_Aberkane_2020}, \cite{Ref_Koch_2020}, \cite{Ref_Sultanow_2020}.

The definitive purpose of this article however is to establish fundamental properties of the "Collatz Feather" and infer provable consequences of those properties to achieve a positive proof of the Collatz conjecture. Our methodology consists of using the complete binary tree and the complete ternary tree\footnote{The complete binary tree over odd numbers is defined as $2\mathbb{N}+1$ endowed with the following two linear applications $\{\cdot2-1;\cdot2+1\}$. The complete ternary tree over odd numbers is defined as $2\mathbb{N}+1$ endowed with operations $\{\cdot3-2;\cdot3;\cdot3+2\}$} over $2\mathbb{N}+1$ as a general coordinate system for each node of the Feather. We owe this strategy to earlier discussions with Feferman \cite{Ref_Feferman_2012} on his investigations on the continuum hypothesis, as it is known the complete binary tree over natural numbers is one way of generating real numbers. The last author's discussion with Feferman argued morphisms, sections and origamis of n-ary trees over $\mathbb{N}$ could be a promising strategy to define objects of intermediate cardinalities between $\aleph _{n}$ and $\aleph _{n+1}$, in a manner inspired from Conway's construction of the surreal numbers \cite{Ref_Knuth_1974}, which itself began by investigating the branching factor of the game of Go. The initial interest therefore, was to investigate the branching factor of the Collatz Feather and to define the cardinality of the set of its branches. Here we begin by identifying and proving five arithmetical rules that apply anywhere on the Collatz dynamic, with the purpose of demonstrating that applying them from number $1$ allows to generate the entire Feather.

\section{Five essential rules of Collatz dynamics}
\label{sec:Fundamental}

\begin{note}
For all intent and purpose we will define \textbf{Syr(x)} or the \textbf{"Syracuse action"} as \textbf{"the next odd number in the forward Collatz orbit of $x$"}. Whenever two numbers $a$ and $b$ have a common number in their orbit, we will also note $a\equiv b$, a relation that is self-evidently transitive:
\[
(a\equiv b) \land (b\equiv c) \Rightarrow a\equiv c
\]
The choice of symbol "$\equiv$" is a deliberate one to acknowledge a kinship between our method and modular arithmetic. 
\end{note}

\begin{definition}{Actions G, V and S}
For any natural number $a$ are specified as follows:
\begin{enumerate}
\item $G(a):=2a-1$
\item $S(a):=2a+1$. The \textbf{rank} of $a$ is its number of consecutive end digits $1$ in base $2$. 
\item $V(a):=4a+1 = G\circ S(a)$
\end{enumerate}
\end{definition}

\newpage
\begin{definition}{Type A, B and C}
\begin{enumerate}
    \item a number $a$ is of type A if its base 3 representation ends with digit 2
    \item a number $b$ is of type B if its base 3 representation ends with digit 0 
    \item a number $c$ is of type C if its base 3 representation ends with digit 1
\end{enumerate}
To remember which is which one need only remember the order of ABC: $a+1$ is dividable by 3, and so is $c-1$, thus A is on the left of B and C is on the right. 
\end{definition}

\begin{figure}[H]
\resizebox{1\textwidth}{!}{
\begin{tikzpicture}
\node[state, draw=purple, fill=purple!10] (1) [ultra thick, below=50pt, left=180pt ] {$1$};
\node[state, draw=gold, fill=gold!10] (3) [ultra thick, below=10pt,  right=20pt ] {$3$};
\node[state, draw=teal, fill=teal!10] (5) [ultra thick, above=30pt, left=180pt ] {$5$};
\node[state, draw=purple, fill=purple!10] (7) [ultra thick, above=30pt, right=130pt ] {$7$};
\node[state, draw=gold, fill=gold!10] (15) [ultra thick, above=80pt, right=180pt ] {$15$};
\node[state, draw=purple, fill=purple!10] (13) [ultra thick, above=80pt, right=20pt ] {$13$};
\node[state, draw=teal, fill=teal!10] (11) [ultra thick, above=80pt, left=130pt ] {$11$};
\node[state, draw=gold, fill=gold!10] (9) [ultra thick, above=80pt, left=290pt ] {$9$};
\node[state, draw=teal, fill=teal!10] (17) [ultra thick, above=130pt, left=350pt ] {$17$};
\node[state, draw=purple, fill=purple!10] (19) [ultra thick, above=130pt, left=270pt ] {$19$};
\node[state, draw=gold, fill=gold!10] (21) [ultra thick, above=130pt, left=180pt ] {$21$};
\node[state, draw=teal, fill=teal!10] (23) [ultra thick, above=130pt, left=100pt ] {$23$};
\node[state, draw=purple, fill=purple!10] (25) [ultra thick, above=130pt, left=10pt ] {$25$};
\node[state, draw=gold, fill=gold!10] (27) [ultra thick, above=130pt, right=40pt ] {$27$};
\node[state, draw=teal, fill=teal!10] (29) [ultra thick, above=130pt, right=130pt ] {$29$};
\node[state, draw=purple, fill=purple!10] (31) [ultra thick, above=130pt, right=200pt ] {$31$};
\path[every node/.style={sloped,anchor=north,auto=false}]
(5) [->] edge [line width=5pt] [draw=iceberg, bend left] node {\large \textcolor{iceberg} G} (9)
(9) [->] edge [line width=5pt] [draw=iceberg, bend left] node {\large \textcolor{iceberg} G} (17)
(1) [->] edge [line width=5pt] [draw=forest, bend left] node {\large \textcolor{forest} V} (5)
(5) [->] edge [line width=5pt] [draw=red, bend right] node {\large \textcolor{red} S} (11)
(5) [->] edge [line width=5pt] [draw=forest, bend left] node {\large \textcolor{forest} V} (21)
(9) [->] edge [line width=5pt] [draw=red, bend right] node {\large \textcolor{red} S} (19)
(11) [->] edge [line width=5pt] [draw=red, bend right] node {\large \textcolor{red} S} (23)
(13) [->] edge [line width=5pt] [draw=iceberg, bend left] node {\large \textcolor{iceberg} G} (25)
(3) [->] edge [line width=5pt] [draw=forest, bend left] node {\large \textcolor{forest} V} (13)
(1) [->] edge [line width=5pt] [draw=red, bend right] node {\large \textcolor{red} S} (3)
(3) [->] edge [line width=5pt] [draw=red, bend right] node {\large \textcolor{red} S} (7)
(7) [->] edge [line width=5pt] [draw=red, bend right] node {\large \textcolor{red} S} (15)
(15) [->] edge [line width=5pt] [draw=red, bend right] node {\large \textcolor{red} S} (31)
(5) [->] edge [line width=5pt] [draw=forest, bend left] node {\large \textcolor{forest} V} (21)
(9) [->] edge [line width=5pt] [draw=red, bend right] node {\large \textcolor{red} S} (19)
(11) [->] edge [line width=5pt] [draw=red, bend right] node {\large \textcolor{red} S} (23)
(13) [->] edge [line width=5pt] [draw=red, bend right] node {\large \textcolor{red} S} (27)
(7) [->] edge [line width=5pt] [draw=forest, bend left] node {\large \textcolor{forest} V} (29)
(3) [->] edge [line width=5pt] [draw=iceberg, bend left] node {\large \textcolor{iceberg} G} (5)
(15) [->] edge [line width=5pt] [draw=iceberg, bend left] node {\large \textcolor{iceberg} G} (29)
(7) [->] edge [line width=5pt] [draw=iceberg, bend left] node {\large \textcolor{iceberg} G} (13)
(11) [->] edge [line width=5pt] [draw=iceberg, bend left] node {\large \textcolor{iceberg} G} (21);
\end{tikzpicture}
}
\caption{Quiver connecting all odd numbers from 1 to 31 with the arrows of actions S,V and G. The set  $2\mathbb{N}+1$ is thus endowed with three unary operations without a general inverse that are non commutative with $G\circ S = V$. Whenever we will mention the inverse of these operations, it will be assuming they exist on $\mathbb{N}$. Type A numbers are circled in teal, B in gold and C in purple.}
\label{fig:1}
\end{figure}
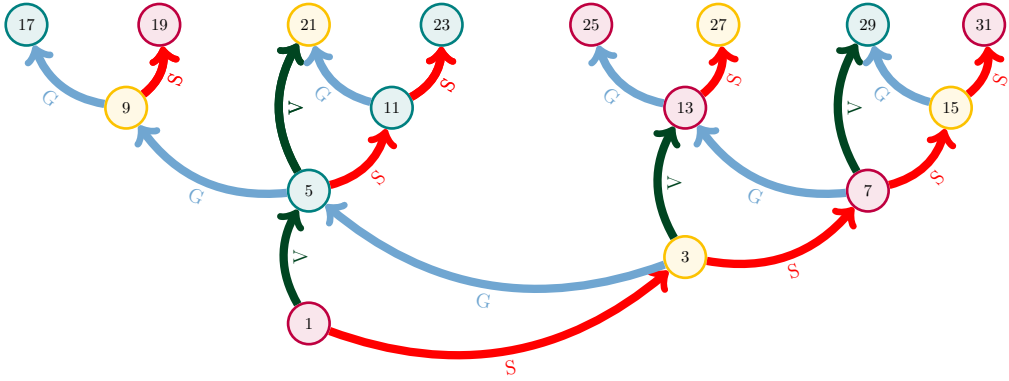

\begin{theorem}
The following arithmetic rules apply anywhere over the system $2\mathbb{N}+1$ endowed with the Collatz dynamic.
\begin{itemize}
\item \textbf{Rule One:} $\forall x$ odd, $V(x)\equiv(x)$
\item \textbf{Rule Two:} $\forall x,k$ odd, $S^{k}V(x)\equiv S^{k+1}V(x)$ and $\forall x,k$ even,  $S^{k}V(x)\equiv S^{k+1}V(x)$
\item \textbf{Rule Three:}

$\forall \{n; y\} \in \mathbb{N}^{2}$, $\forall x$ odd non B, $3^{n}x\equiv y\Rightarrow \bigwedge\limits_{i=1}^n (V(4^{i}3^{n-i}x))\wedge S(V(4^{i}3^{n-i}x)) \equiv y$
\item \textbf{Rule Four:}

$\forall \{n;y\} \in \mathbb{N}^{2}$ , $\forall x$ odd non B, $S(3^{n}x)\equiv y\Rightarrow \bigwedge\limits_{i=1}^n (S(4^{i}3^{n-i}x)\wedge S^{2}(4^{i}3^{n-i}x))\equiv y$
        
\item \textbf{Rule Five} 

$\forall n \in \mathbb{N}$, $\forall y \in \mathbb{N}$, $\forall x$ odd non B where $3^{n}x$ is of rank 1, $a\equiv y$, $a=G(3^{n}x)\\
\Rightarrow \bigwedge\limits_{i=0}^n (S^{i}(G(3^{n-i}x)) \wedge S^{i+1}(G(3^{n-i}x))) \equiv y$
\end{itemize}
\end{theorem}

In the following we will prove these five rules.
\begin{note}
In reference to \textbf{Figure 1} we will call \textbf{"vertical even"} a number that can be written $V(e)$ where $e$ is even, and \textbf{"vertical odd"} if it can be written $V(o)$ where $o$ is odd. For example, $9$ is the first vertical even number and $5$ is the first vertical odd. 
\end{note}

\subsection{Proving Rule One}
If $a$ is written $4b+1$ then $3a+1 = 12b+4 = 4(3b+1)$ therefore $a\equiv b$.

\subsection{Proving Rule Two}
\begin{lemma}
Let $a$ be a number of rank $1$ thus with an odd number $p$ so that $a=G(p)$ then $Syr(S(a))=G(3\cdot p)$. Let $a$ be a number of rank $n$ so that $S^{-(n-1)}(a)=G(p)$ then $Syr^{n-1}(a)=G(3^{n-1}\cdot p)$
\end{lemma}

\begin{proof}
If $a = 2p-1$, $p$ is odd, then it follows:
\[
{\renewcommand{\arraystretch}{1.8}
\begin{array}{l}
S(a)=4p-1\\
\frac{3\cdot S S(a)+1}{2} = \frac{12p-2}{2} = 6p-1 = G(3\cdot p)\\
Syr(S(a))= G(3\cdot p)
\end{array}}
\]

Let's generalize to the $n$. If $Syr(S(a))$ can be written $G(3\cdot p)$ it is also of rank 1, whereas $S(a)$ was of rank 2, therefore, the Syracuse action has made it lose one rank. All we have to prove now is that $Syr(S^2(a))$ = $S(Syr(S(a)))$ under those conditions:
\[
{\renewcommand{\arraystretch}{1.8}
\begin{array}{l}
\frac{3\cdot(S^2(a))+1}{2}= 6a+5\\
S(Syr(S(a)))=S(3a+2)=6a+5=Syr(S^2(a))
\end{array}}
\]

If a is of rank $n>1$, $Syr(a)$ is of rank $n-1$, and $Syr(S(a)) = S(Syr(a))$
\end{proof}

\begin{note}
The $3n+1$ action over an odd number, since it necessary yields an even one, is \textit{in fine} equivalent to adding $1$ to it, then the half of the result, then $-1$. How many times one can add an half to an odd number $+1$ directly depends on its base 2 representation, and in particular its number of consecutive end digits 1. Let us take Mersenne numbers for example, which are defined as $2^n-1$. One can transform them consecutively in this way a number of time equal to their rank-1, indeed, $31$, which is written $11111$ in base 2 is of rank $5$, because $32=2^5$ so if one repeats the action "add to the number+1 the half of itself" this will yield an even result \textbf{exactly four consecutive times}. Thus, any strictly ascending Collatz orbit concerns only numbers $a$ of rank $n>1$, and is defined by
\[
(a+1)\cdot \left(\frac{3}{2}\right)^{n-1}-1
\]
\end{note}

\begin{lemma}
Let $a$ be an odd number of rank $1$ that is vertical even, then $3a$ is of rank 2 or more, and $9a$ is vertical even. Let $a$ be an odd number of rank $1$ that is vertical odd, then $3a$ is of rank $2$ or more, and $9a$ is vertical odd.
\end{lemma}
 
\begin{proof}
If $a$ is vertical even it can be written $8k+1$ $\forall k:3a=24k+3$ and this number admits an $S^{-1}$ that is $12k+1$, which is an odd number, therefore $3a$ is at least of rank $2$.

Moreover, $9a=72k+9$ and this number admits a $V^{-1}$ that is $18k+2$, an even number. Now if $a$ is vertical odd, it can be written $8k+5$ and $\forall k:3a=24k+15$ and $9a=72k+45$. It follows that $3a$ admits an $S^{-1}$ and $9a$ admits a $V^{-1}$, respectively $12k+7$ and $18k+11$ and they are both odd.
\end{proof}

\begin{lemma}
Let $a$ be a number that is vertical even, then $(a) \equiv S(a)$ and $S^k(a)\equiv S^{k+1}(a)$ for any even k.
Let $a$ be a number that is vertical odd, then $S(a) \equiv S^2(a)$ and $S^k(a)\equiv S^{k+1}(a)$ for any odd k.
\end{lemma}
 
\begin{proof}
If $a$ is vertical even then it can be written as $G(p)$ where $p$ is necessarily vertical (odd or even). We proved that $3p$ is then of rank $2$ or more and also that we have $Syr(S(a))=G(p)$ so it is necessarily vertical odd (since $3d$ is of rank $2$ or more) so $Syr(a)=V^{-1}(Syr(S(a))$ and therefore $a\equiv S(a)$. This behavior we can now generalize to the $n$, because if $a$ is vertical even with $a=G(p)$, then the lemmas we used also provide that $Syr^n(S^n(a))=G(3^n\cdot p)$ and therefore $Syr^n(S^n(a))$ will be vertical even for any even $n$ because $3^n\cdot p$ will be vertical something (even or odd, depending on $p$ only) for any even $n$.

Now if $a$ is vertical odd it can be written $G(p)$ and $p$ is necessarily of rank $2$ or more because $G\circ S = V$. Thus $3p$ is vertical (even or odd) and therefore $Syr(S(a))=G(3p)$ is vertical even.
\end{proof}

\begin{note}
Observe that in the process of proving \textbf{Rule Two} we also demonstrated that any number of rank $2$ or more is finitely turned into a rank $1$ number of type A by the Collatz dynamic, and that any number $x$ of rank $2$ or more so that $x\equiv S(x)$ under \textbf{Rule Two} is finitely mapped to a type A number that is vertical even, therefore \textbf{proving the convergence of such numbers is enough to prove the Collatz Conjecture}.
\end{note}

\subsection{Proving Rules Three and Four}
\begin{lemma}
Let $a$ be a vertical even number with $a=G^{n+2}(S(b))$ where $n$ and $b$ are odd, then $a\equiv 3^\frac{n+1}{2}(b)$. Let $a$ be a vertical even number with $a=G^{m+2}(S(b))$ where $m$ is even (zero included) and $b$ is odd, then $a\equiv S(3^\frac{m}{2}(b))$
\end{lemma}

\begin{proof}
If $a=G^{n+2}(S(b))$ by definition $a=2^{n+3}b+1$. Then $3\cdot a+1=3(2^{n+3}b+1)+1)=2^{n+3}\cdot(3b)+4.$ As this expression can be divided by $2$ no more than twice, we have $Syr(a)=2^{n+1}3b+1 = G^{n}(S(3b))$.

Note that if $n=1$ then $V^{-1}(Syr(a))=V^{-1}(2^{2}\cdot(3a)+1)=2^{2}\cdot\frac{1}{4}\cdot(3b)=3b$ which is of course an odd number. Therefore $Syr(a)$ is vertical odd and $V^{-1}(Syr(a))=3b$ thus we have proven that $a\equiv3b$.

If $n=0$ then $a=2^3\cdot b+1$ so $3(a+1)=2^3\cdot{3b}+4$ therefore $Syr(a)=S(3b)$ and thus $a \equiv S(3b)$. From this we can generalise the progression of numbers that can be written $G^{n}(x)$ where $x$ is of rank $2$ or more. Let $b$ be any odd number:

\begin{itemize}
\item All \textbf{"Variety S"} numbers above $b$ are written $V(b\cdot2^{2k-1})\;\; \text{or}\;\; S(b\cdot2^{2k}) = 2^{2k+1}\cdot b+1$ and
\item all \textbf{"Variety V"} numbers above $b$ are written $V(b\cdot4^{k})$ or equivalently $S(b\cdot2^{2k+1}) = 4^{k+1}\cdot b+1$.
\end{itemize}

Any number $g$ that can be written $G^{n}(V(x))$ with $x$ odd and $n>0$ may thus be finitely reduced under the Collatz dynamic to a number that can be written either $S(3^{m}x)$ or $V(3^{m}x)$ by the repeated following transformation: 
\[
(g-1)\cdot\biggl(\frac{3}{4}\biggr)^{k}+1
\]

Therefore we have indeed that,
\begin{itemize}
\item for Variety S numbers: $2^{2k+1}\cdot b\cdot\left(\frac{3}{4}\right)^{k}+1=2b\cdot3^{k}+1=S(b\cdot3^{k})$, which proves \textbf{Rule Four}.
\item for Variety V numbers: $4\cdot4^{k}\cdot b\cdot\left(\frac{3}{4}\right)^{k}+1=4b\cdot3^{k}+1=V(b\cdot3^{k})$ which proves \textbf{Rule Three} because \textbf{Rule One} already provides that $V(b\cdot3^{k}) \equiv b\cdot3^{k}$.
\end{itemize}
\end{proof}

\subsection{Proving Rule Five}
Any type A number of rank 1 can be written $a=G(b)$ where $b$ is of type B. In proving \textbf{Rule Two} we showed that any number of rank $n>1$ is finitely mapped by the Collatz dynamics to $G(3^{n-1}\cdot G^{-1}(S^{-(n-1)}(a)))$, which combined with \textbf{Rule Two} itself gives \textbf{Rule Five}.

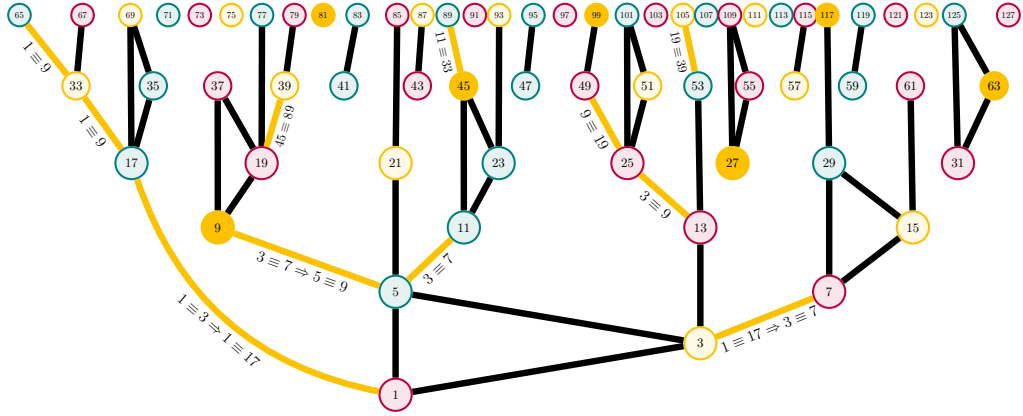
\begin{figure} [H]
\resizebox{1\textwidth}{!} {
\begin{tikzpicture} 
\node[state, draw=purple, fill=purple!10] (1) [ultra thick, below=50pt, left=180pt ] {$1$};
\node[state, draw=gold, fill=gold!10] (3) [ultra thick, below=10pt,  right=30pt ] {$3$};
\node[state, draw=teal, fill=teal!10] (5) [ultra thick, above=30pt, left=180pt ] {$5$};
\node[state, draw=purple, fill=purple!10] (7) [ultra thick, above=30pt, right=130pt ] {$7$};
\node[state, draw=gold, fill=gold!10] (15) [ultra thick, above=80pt, right=195pt ] {$15$};
\node[state, draw=purple, fill=purple!10] (13) [ultra thick, above=80pt, right=30pt ] {$13$};
\node[state, draw=teal, fill=teal!10] (11) [ultra thick, above=80pt, left=127pt ] {$11$};
\node[state, draw=gold, fill=gold] (9) [ultra thick, above=80pt, left=318pt ] {$9$};
\node[state, draw=teal, fill=teal!10] (17) [ultra thick, above=130pt, left=385pt ] {$17$};
\node[state, draw=purple, fill=purple!10] (19) [ultra thick, above=130pt, left=284pt ] {$19$};
\node[state, draw=gold, fill=gold!10] (21) [ultra thick, above=130pt, left=180pt ] {$21$};
\node[state, draw=teal, fill=teal!10] (23) [ultra thick, above=130pt, left=100pt ] {$23$};
\node[state, draw=purple, fill=purple!10] (25) [ultra thick, above=130pt, left=0pt ] {$25$};
\node[state, draw=gold, fill=gold] (27) [ultra thick, above=130pt, right=55pt ] {$27$};
\node[state, draw=teal, fill=teal!10] (29) [ultra thick, above=130pt, right=130pt ] {$29$};
\node[state, draw=purple, fill=purple!10] (31) [ultra thick, above=130pt, right=230pt ] {$31$};
\node[state, draw=gold, fill=gold!10, minimum size=0.7pt] (33) [ultra thick, above=190pt, left=430pt ] {$33$};
\node[state, draw=teal, fill=teal!10, minimum size=0.7pt] (35) [ultra thick, above=190pt, left=370pt ] {$35$};
\node[state, draw=purple, fill=purple!10, minimum size=0.7pt] (37) [ultra thick, above=190pt, left=320pt ] {$37$};
\node[state, draw=gold, fill=gold!10, minimum size=0.7pt] (39) [ultra thick, above=190pt, left=268pt ] {$39$};
\node[state, draw=teal, fill=teal!10, minimum size=0.7pt] (41) [ultra thick, above=190pt, left=222pt ] {$41$};
\node[state, draw=purple, fill=purple!10, minimum size=0.7pt] (43) [ultra thick, above=190pt, left=165pt ] {$43$};
\node[state, draw=gold, fill=gold, minimum size=0.7pt] (45) [ultra thick, above=190pt, left=129pt ] {$45$};
\node[state, draw=teal, fill=teal!10, minimum size=0.7pt] (47) [ultra thick, above=190pt, left=81pt ] {$47$};
\node[state, draw=purple, fill=purple!10, minimum size=0.7pt] (49) [ultra thick, above=190pt, left=35pt ] {$49$};
\node[state, draw=gold, fill=gold!10, minimum size=0.7pt] (51) [ultra thick, above=190pt, right=-9pt ] {$51$};
\node[state, draw=teal, fill=teal!10, minimum size=0.7pt] (53) [ultra thick, above=190pt, right=30pt ] {$53$};
\node[state, draw=purple, fill=purple!10, minimum size=0.7pt] (55) [ultra thick, above=190pt, right=70pt ] {$55$};
\node[state, draw=gold, fill=gold!10, minimum size=0.7pt] (57) [ultra thick, above=190pt, right=105pt ] {$57$};
\node[state, draw=teal, fill=teal!10, minimum size=0.7pt] (59) [ultra thick, above=190pt, right=150pt ] {$59$};
\node[state, draw=purple, fill=purple!10, minimum size=0.7pt] (61) [ultra thick, above=190pt, right=195pt ] {$61$};
\node[state, draw=gold, fill=gold, minimum size=0.7pt] (63) [ultra thick, above=190pt, right=260pt ] {$63$};
\node[state, draw=teal, fill=teal!10, scale=0.7] (65) [ultra thick, above=350pt, left=680pt ] {$65$};
\node[state, draw=purple, fill=purple!10, scale=0.7] (67) [ultra thick, above=350pt, left=610pt ] {$67$};
\node[state, draw=purple, fill=purple!10, scale=0.7] (127) [ultra thick, above=350pt, right=390pt ] {$127$};
\node[state, draw=gold, fill=gold!10, scale=0.7] (69) [ultra thick, above=350pt, left=557pt ] {$69$};
\node[state, draw=teal, fill=teal!10, scale=0.7] (71) [ultra thick, above=350pt, left=515pt ] {$71$};
\node[state, draw=purple, fill=purple!10, scale=0.7] (73) [ultra thick, above=350pt, left=480pt ] {$73$};
\node[state, draw=gold, fill=gold!10, scale=0.7] (75) [ultra thick, above=350pt, left=445pt ] {$75$};
\node[state, draw=teal, fill=teal!10, scale=0.7] (77) [ultra thick, above=350pt, left=411pt ] {$77$};
\node[state, draw=purple, fill=purple!10, scale=0.7] (79) [ultra thick, above=350pt, left=375pt ] {$79$};
\node[state, draw=gold, fill=gold, scale=0.7] (81) [ultra thick, above=350pt, left=343pt ] {$81$};
\node[state, draw=teal, fill=teal!10, scale=0.7] (83) [ultra thick, above=350pt, left=305pt ] {$83$};
\node[state, draw=purple, fill=purple!10, scale=0.7] (85) [ultra thick, above=350pt, left=261pt ] {$85$};
\node[state, draw=gold, fill=gold!10, scale=0.7] (87) [ultra thick, above=350pt, left=233pt ] {$87$};
\node[state, draw=teal, fill=teal!10, scale=0.7] (89) [ultra thick, above=350pt, left=205pt ] {$89$};
\node[state, draw=purple, fill=purple!10, scale=0.7] (91) [ultra thick, above=350pt, left=175pt ] {$91$};
\node[state, draw=gold, fill=gold!10, scale=0.7] (93) [ultra thick, above=350pt, left=148pt ] {$93$};
\node[state, draw=teal, fill=teal!10, scale=0.7] (95) [ultra thick, above=350pt, left=110pt ] {$95$};
\node[state, draw=purple, fill=purple!10, scale=0.7] (97) [ultra thick, above=350pt, left=75pt ] {$97$};
\node[state, draw=gold, fill=gold, scale=0.7] (99) [ultra thick, above=350pt, left=40pt ] {$99$};
\node[state, draw=teal, fill=teal!10, scale=0.7] (101) [ultra thick, above=350pt, left=6pt ] {$101$};
\node[state, draw=purple, fill=purple!10, scale=0.7] (103) [ultra thick, above=350pt, left=-26pt ] {$103$};
\node[state, draw=gold, fill=gold!10, scale=0.7] (105) [ultra thick, above=350pt, left=-56pt ] {$105$};
\node[state, draw=teal, fill=teal!10, scale=0.7] (107) [ultra thick, above=350pt, left=-82pt ] {$107$};
\node[state, draw=purple, fill=purple!10, scale=0.7] (109) [ultra thick, above=350pt, left=-108pt ] {$109$};
\node[state, draw=gold, fill=gold!10, scale=0.7] (111) [ultra thick, above=350pt, left=-136pt ] {$111$};
\node[state, draw=teal, fill=teal!10, scale=0.7] (113) [ultra thick, above=350pt, left=-165pt ] {$113$};
\node[state, draw=purple, fill=purple!10, scale=0.7] (115) [ultra thick, above=350pt, left=-192pt ] {$115$};
\node[state, draw=gold, fill=gold, scale=0.7] (117) [ultra thick, above=350pt, left=-216pt ] {$117$};
\node[state, draw=teal, fill=teal!10, scale=0.7] (119) [ultra thick, above=350pt, left=-256pt ] {$119$};
\node[state, draw=purple, fill=purple!10, scale=0.7] (121) [ultra thick, above=350pt, left=-292pt ] {$121$};
\node[state, draw=gold, fill=gold!10, scale=0.7] (123) [ultra thick, above=350pt, left=-326pt ] {$123$};
\node[state, draw=teal, fill=teal!10, scale=0.7] (125) [ultra thick, above=350pt, left=-357pt ] {$125$};
\path[every node/.style={sloped,anchor=north,auto=false}]
(1) edge [line width=5pt] node {} (3)
(1) edge [line width=5pt] node {} (5)
(3) edge [line width=5pt] node {} (5)
(3) edge [line width=5pt] node {} (13)
(7) edge [line width=5pt] node {} (15)
(7) edge [line width=5pt] node {} (29)
(5) edge [line width=5pt] node {} (21)
(9) edge [line width=5pt] node {} (19)
(11) edge [line width=5pt] node {} (23)
(15) edge [line width=5pt] node {} (29)
(17) edge [line width=5pt] node {} (35)
(9) edge [line width=5pt] node {} (37)
(19) edge [line width=5pt] node {} (37)
(11) edge [line width=5pt] node {} (45)
(23) edge [line width=5pt] node {} (45)
(25) edge [line width=5pt] node {} (51)
(13) edge [line width=5pt] node {} (53)
(27) edge [line width=5pt] node {} (55)
(15) edge [line width=5pt] node {} (61)
(31) edge [line width=5pt] node {} (63)
(17) edge [line width=5pt] node {} (69)
(35) edge [line width=5pt] node {} (69)
(19) edge [line width=5pt] node {} (77)
(21) edge [line width=5pt] node {} (85)
(23) edge [line width=5pt] node {} (93)
(25) edge [line width=5pt] node {} (101)
(51) edge [line width=5pt] node {} (101)
(27) edge [line width=5pt] node {} (109)
(55) edge [line width=5pt] node {} (109)
(29) edge [line width=5pt] node {} (117)
(31) edge [line width=5pt] node {} (125)
(63) edge [line width=5pt] node {} (125)
(33) edge [line width=5pt] node {} (67)
(39) edge [line width=5pt] node {} (79)
(41) edge [line width=5pt] node {} (83)
(43) edge [line width=5pt] node {} (87)
(47) edge [line width=5pt] node {} (95)
(49) edge [line width=5pt] node {} (99)
(57) edge [line width=5pt] node {} (115)
(59) edge [line width=5pt] node {} (119)
(1) [-] edge [line width=5pt, draw=gold]  [bend left] node {\large{$1\equiv3 \Rightarrow 1\equiv17$}}(17)
(3) [-] edge [line width=5pt, draw=gold]   node {\large{$1\equiv17 \Rightarrow 3\equiv7$}}(7)
(5) [-] edge [line width=5pt, draw=gold]   node {\large{$3\equiv7 \Rightarrow 5\equiv9$}}(9)
(17) [-] edge [line width=5pt, draw=gold]   node {\large{$1\equiv9$}}(33) 
(33) [-] edge [line width=5pt, draw=gold]   node {\large{$1\equiv9$}}(65) 
(13) [-] edge [line width=5pt, draw=gold]   node {\large{$3\equiv9$}}(25) 
(25) [-] edge [line width=5pt, draw=gold]   node {\large{$9\equiv19$}}(49) 
(5) [-] edge [line width=5pt, draw=gold]   node {\large{$3\equiv7$}}(11) 
(45) [-] edge [line width=5pt, draw=gold]   node {$11\equiv33$}(89) 
(19) [-] edge [line width=5pt, draw=gold]   node {$45\equiv89$}(39) 
(53) [-] edge [line width=5pt, draw=gold]   node {$19\equiv39$}(105);
\end{tikzpicture}
}
\caption{Just a few applications of \textbf{Rules Three, Four and Five} starting from $1\equiv3\equiv5$ are here plotted in gold. \textbf{Rules One and Two} are plotted in black. Whenever a number is connected to $1$ by a finite path of black and/or gold edges it is proven to converge to $1$.}
\label{fig:2}
\end{figure}

\section{The Golden Automaton}
\begin{definition}
On \{$2\mathbb{N}+1$; G, S\} where \textbf{Rules One and Two} are considered pre-computed (the black edges on Figure 2) the systematic computation of \textbf{Rules Three, Four and Five} from number $1$ onward is called the \textbf{"Golden Automaton"}.
\end{definition}

\subsection{"Golden Arithmetic"}
Our  purpose is to develop an \textit{ad hoc} unary algebra that could found a congruence arithmetic specifically made to prove the Collatz conjecture, and which we intend as an epistemological extension of modular arithmetic, hence our use of the symbol $\equiv$ in this article rather than the usual $\thicksim$ which is seen more frequently in the Collatz-related literature. This \textbf{"Golden arithmetic"} involves words taken in the alphabet $\{G; S; V; 3\}$, which we will call in their order of application, just like in turtle graphics. For example VGS3 means $3\cdot S\circ G \circ V$

Rules 3, 4 and 5 may now be reformulated as such, without loss of generality as long as \textbf{Rules One and Two} are still assumed:
\begin{itemize}
\item \textbf{Rule Three:} Let b be of type B, then $b\equiv  VGS3^{-1}$ from b. We will call this action $R_b(x)=16\frac{x}{3}+1$
\item \textbf{Rule Four:} Let c be of type C, then $c\equiv GS3^{-1}$ from c. We will call this action $R_c(x)=\frac{4x-1}{3}$
\item \textbf{Rule Five:} Let a be of type A, then $a\equiv G3^{-1}$ from a. We will call this action $R_a(x)=\frac{2x-1}{3}$
\end{itemize}

As \textbf{Rules One and Two} ensure that the quiver generated by the Golden Automaton is branching, with each type B number that is vertical even providing both a new A type and a new B type number to keep applying respectively rules 5 and 3, we may follow only the pathway of type A numbers to define a single non-branching series of arrows, forming a single infinite branch of the quiver. The latter, if computed from number 15, leads straight to 31 and 27, solving a great deal of other numbers on the way:
{\setlength{\mathindent}{13em}\[
\arraycolsep=0.4pt
\begin{array}{rll}
15   &\equiv81    &\hspace{2em}\text{Rule 3}\\
81   &\equiv1025  &\hspace{2em}\text{First type A reached by \textbf{Rule 3}} \\
1025 &\equiv303   &\hspace{2em}\textbf{Rule 5}\\ 
303  &\equiv607   &\hspace{2em}\text{Rule 2}\\ 
607  &\equiv809   &\hspace{2em}\text{Rule 4}\\
809  &\equiv159   &\hspace{2em}\textbf{Rule 5}\\
159  &\equiv319   &\hspace{2em}\text{Rule 2}\\
319  &\equiv425   &\hspace{2em}\text{Rule 4}\\
425  &\equiv283   &\hspace{2em}\textbf{Rule 5}\\ 
283  &\equiv377   &\hspace{2em}\text{Rule 4}\\ 
377  &\equiv111   &\hspace{2em}\textbf{Rule 5}\\ 
111  &\equiv593   &\hspace{2em}\text{Rule 3}\\
593  &\equiv175   &\hspace{2em}\textbf{Rule 5}\\ 
175  &\equiv233   &\hspace{2em}\text{Rule 4}\\
233  &\equiv103   &\hspace{2em}\textbf{Rule 5}\\ 
103  &\equiv137   &\hspace{2em}\text{Rule 4}\\
137  &\equiv91    &\hspace{2em}\textbf{Rule 5}\\ 
91   &\equiv161   &\hspace{2em}\text{Rule 4}\\
161  &\equiv31    &\hspace{2em}\textbf{Rule 5}\\ 
31   &\equiv41    &\hspace{2em}\text{Rule 4}\\
41   &\equiv27    &\hspace{2em}\textbf{Rule 5} 
\end{array}
\]}

Again, it is in no way a problem, but rather a powerful property of the Golden Automaton that this particular quiver branch already cover 19 steps (and actually more) because each of them is branching into other solutions.

We may follow another interesting sequence to show that in the same way that Mersenne number 15 finitely solves Mersenne number 31, Mersenne number 7 solves Mersenne number 127, this time we will follow a B branch up to $Syr^{6}(127)$ which we know can be written $G(3^6)$ because 127 is the Mersenne of rank 7.

By Rule 4 we have the first equivalence $\boldsymbol{7\equiv9}$ and $\boldsymbol{9\equiv25\equiv49}$.

So by Rule 2 we also have $\boldsymbol{25\equiv51}$.

Rule 3 gives $\boldsymbol{51\equiv273}$ and again $\boldsymbol{273\equiv1457=G(729)\equiv127}$.

The cases of $15$ proving the convergence of $31$ and $27$ and of $7$ proving the one of $127$ naturally lead us to the following conjecture:

\begin{conjecture}
Suppose all odd numbers up to $2^{n}$ are proven to converge to $1$ under the Collatz dynamic, then the Golden Automaton finitely proves the convergence of those up to $2^{n+1}$
\end{conjecture}

And indeed we already have that the Golden Automaton starting with $1$ proves $3$ by \textbf{Rule One}, then $3$ proves all numbers from $5$ to $15$ which in turn prove all numbers from $33$ to $127$. In the next subsection we render larger quivers generated by the Golden Automaton to provide a better understanding of their geometry and fundamental properties, and to demonstrate why it is so, and more generally, why it can be proven they can reach any number in $2\mathbb{N}+1$ in ZFC. 

\subsection{Computational Scale-up}
The purpose of this subsection is to identify provable fundamental properties of the Golden Automaton by scaling it up on the full binary tree over $2\mathbb{N}+1$. To streamline its algorithmic scaling, we use the simplified rules we defined in the previous subsection, again, without loss of generality. Our precise purpose is to pave the way for a formal demonstration that proving the convergence of odd numbers up to $n$ is always isomorphic to a Hydra Game. In the next figures we color all the elements of $24\mathbb{N}-7$ as for example $\{17,41,65,\ldots\}$ in red to as we demonstrate in the next section they precisely from the "heads" in the Hydra Game. 

\begin{figure}[H]
\includegraphics[clip, trim=6cm 0cm 6cm 0cm, width=1.00\textwidth, page=1]{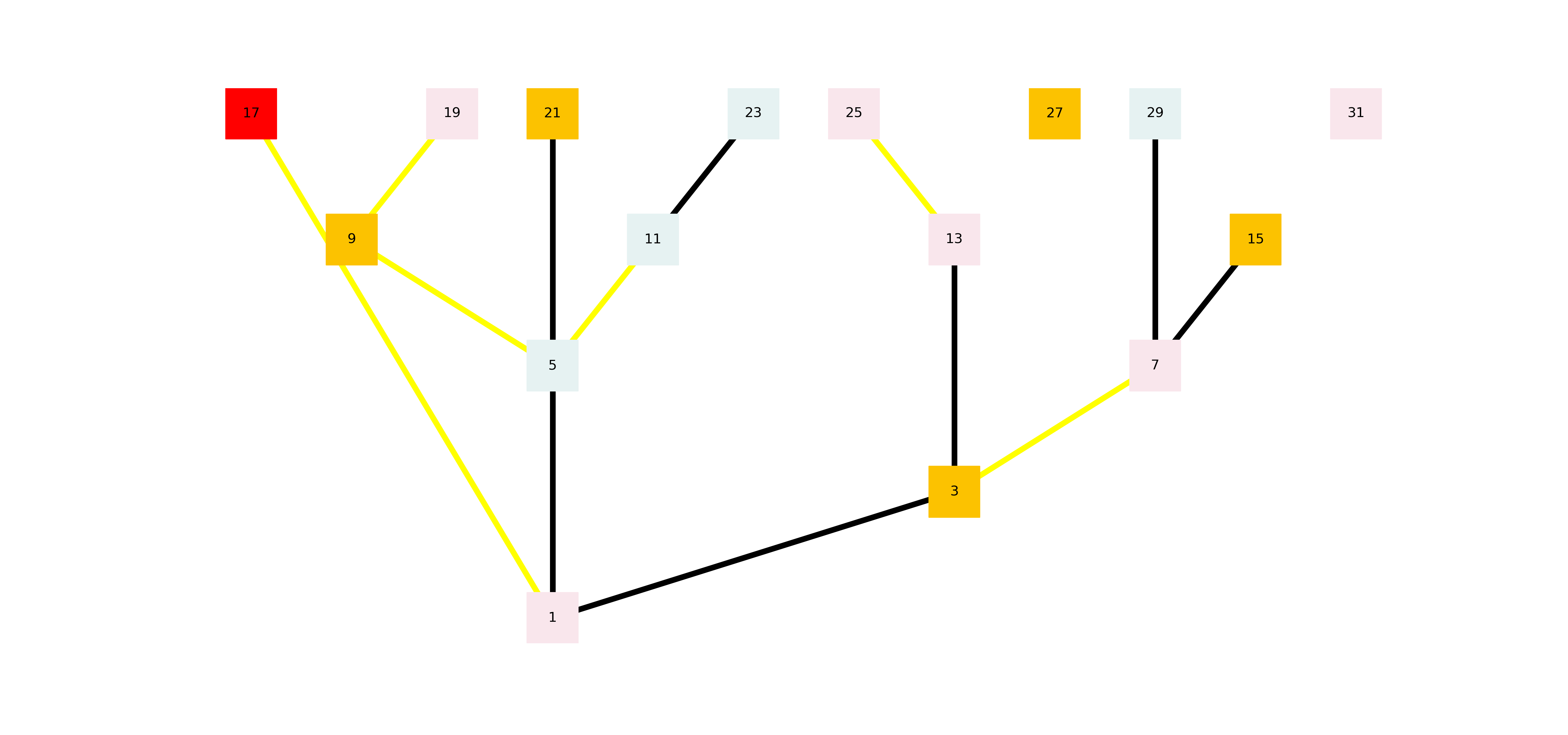}
\caption{Golden Automaton confined to numbers smaller than 32}
\label{fig:3}
\end{figure}

\begin{figure}[H]
\includegraphics[clip, trim=9cm 0cm 8cm 0cm, width=1.00\textwidth, page=1]{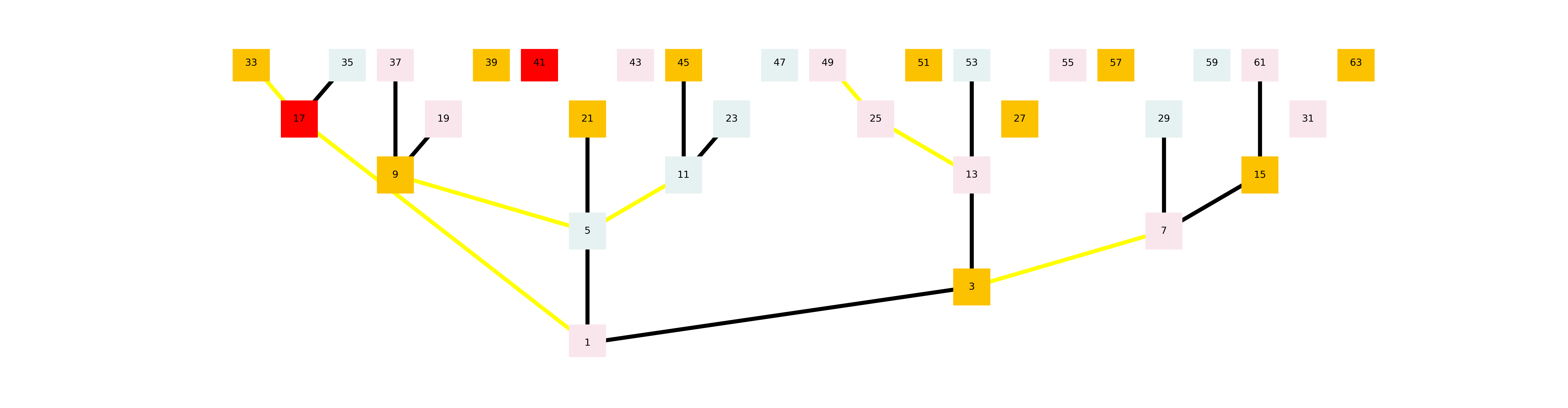}
\caption{Golden Automaton confined to numbers smaller than 64}
\label{fig:4}
\end{figure}

\begin{figure}[H]
\includegraphics[clip, trim=13cm 0cm 12cm 0cm, width=1.00\textwidth, page=1]{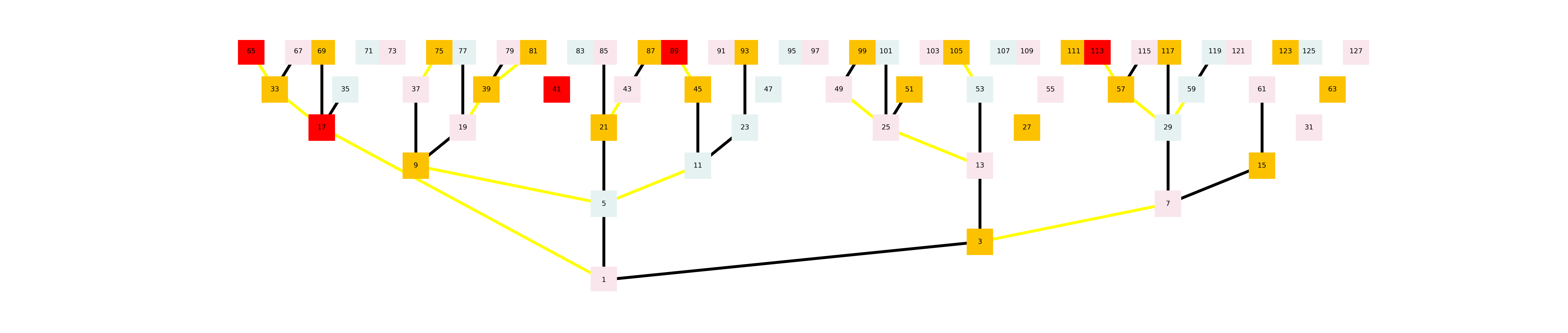}
\caption{Golden Automaton confined to numbers smaller than 128}
\label{fig:5}
\end{figure}

\begin{figure}[H]
\includegraphics[clip, trim=30cm 0cm 28cm 0cm, width=1.00\textwidth, page=1]{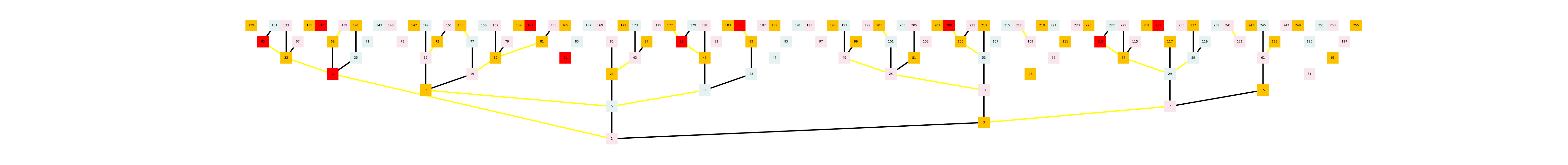}
\caption{Golden Automaton confined to numbers smaller than 256}
\label{fig:6}
\end{figure}

\section{Proving the Collatz conjecture}
In this section we prove that the Collatz conjecture holds in the Zermelo–Fraenkel set theory with the axiom of choice included (abbreviated ZFC). And we prove that Collatz conjecture also holds in the Peano Arithmetic.

\subsection{ZFC proves the Collatz conjecture}
\begin{theorem}
In ZFC, the Collatz conjecture is true. 
\end{theorem}

\begin{definition}
A \textbf{hydra} is a rooted tree with arbitrary many and arbitrary long finite branches. Leaf nodes are called \textbf{heads}. A head is \textbf{short} if the immediate parent of the head is the root; \textbf{long} if it is neither short nor the root. The object of the \textbf{Hydra game} is to cut down the hydra to its root. At each step, one can cut off one of the heads, after which the hydra grows new heads according to the following rules:
\begin{itemize}
\item if the head was long, grow $n$ copies of the subtree of its parent node minus the cut head, rooted in the grandparent node. 
\item if the head was short, grow nothing
\end{itemize}
\end{definition}

\begin{lemma}
The Golden Automaton reaching any natural number is a Hydra Game over a finite subtree of the complete binary tree over $24\mathbb{N}-7$.
\end{lemma}

\begin{proof}
The essential questions to answer in demonstrating either a homomorphism between a Hydra game and the Golden Automaton reaching any odd number, or that the Golden Automaton is playing \textbf{at worst} a Hydra Game are:
\begin{itemize}
\item What are the Hydra's heads?
\item How do they grow? 
\item Does the Golden Automaton cut them according to the rules (at worst)? 
\end{itemize}

\begin{definition}
A type A number that is vertical even is called an $A_g$. The set of $A_g$ numbers is $24\mathbb{N}-7$. Type B numbers that verify $b\equiv S(b)$ and type C numbers that verify $c\equiv S(c)$ under \textbf{Rule Two} are called \textbf{Bups} and \textbf{Cups} respectively.
\end{definition}

\textbf{What are the Hydra's heads?}
$A_{g}$ numbers are the heads of the Hydra. They are 12 points apart on $2\mathbb{N}+1$ (24 in nominal value, e.g. 17 to 41) and any Bup or Cup of $rank>1$ they represent under \textbf{Rule Five} is smaller than them since action $R_a$ is strictly decreasing so up to the $n^{th}$  $A_{g}$ there are $2n$ (Bups + Cups) of rank 2 or more and half of them are equivalent to these $A_{g}$ (e.g. between 17 and 41 Bup 27 is equivalent to $A_g$ 41, which is equivalent to Cup 31 by \textbf{Rule Four}

\textbf{How do they grow?}
Between any two consecutive $A_g$ in $2\mathbb{N}+1$ there are 
\begin{itemize}
\item 8 non-A numbers
\item 1 of them at most is mapped to the second $A_g$
\item 3 at most are "ups" (Bup or Cup) of rank 2 or more
\end{itemize}

Besides, we also have anywhere that:
\begin{itemize}
\item  Let b be of type B, there are $\frac{2b}{3}$ numbers of type $A_g$ that are smaller than $V^2(b)$ 
\item  Let c be of type C, there are $\frac{S(c)}{3}$ numbers of type $A_g$ that are smaller than $V^2(c)$ 
\item  Let 3c be a type B where c is of type C, there are $\frac{S(c)}{3}$ numbers of type $A_g$ up to $R_b(3c)$ included
\item  Let 3a be a type B where a is of type A, there are $\frac{G(a)}{3}$ numbers of type $A_g$ smaller than $R_b(3a)$
\end{itemize}

Which is defining the growth of the heads. Indeed, any supposedly diverging $A_g$ is forming a Hydra, as we have proven $24\mathbb{N}-7$ contains an image of all undecided Collatz numbers and that any non-decreasing trajectory identifies a subtree within this set. 

\textbf{Does the Golden Automaton play a Hydra game?}

It could be demonstrated that the Golden Automaton is playing an even simpler game as it is branching and thus cutting heads several at a time and in particular cutting some long heads without them doubling \footnote{In fact, the reason the Golden Automaton dominates $24\mathbb{N}-7$ so fast is that it is playing a much simpler game one could call "Hecatonchire v. Hydra game" namely a Hydra game where Herakles' number of arms is also multiplying at each step} but as this is needless for the final proof we can now simply demonstrate that even under the worst possible assumptions it follows at least the rules of a regular Hydra game. 

The computing of $15\equiv\ldots\equiv 27$ that we detailed in \textbf{Subsection 3.1} is one case of the playing of a Hydra Game by the Golden Automaton; we underlined each use of \textbf{Rule 5} specifically so the reader can now report to it more easily, because each time this rule is used, a head ($A_g$) has just been cut. 

The demonstration that $27$ and $31$ converge is the cutting of heads $41$ and $161$ respectively. This single branch of the Automaton having first cut head $17$, reaches to the head $1025$ via B-typed numbers $15$ and $81$. It is therefore playing a Hydra with $\frac{1025+7}{24}=43$ heads of which one ($17$) is already cut at this point and of which at least $8$ are rooted (so cutting them does not multiply any number of heads). This process being independent of the targeted number, we now have that the reaching of any number by the Golden Automaton is at least equivalent to the playing of a Hydra with $n$ heads of which $0<m<n$ are rooted. Even without demonstrating more precise limit theorems for each factors $n$ and $m$ (which could still be a fascinating endeavor) the road is now open for a final resolution of the Collatz conjecture. 
\end{proof}

From there indeed, we know with Goodstein \cite{Ref_Goodstein_1944} and Kirby and Paris \cite{Ref_Kirby_1982} that assuming a system strong enough to prove Peano arithmetic is consistent and that $\epsilon _0$ is well-ordered, no Hydra game can be lost. Since we have that the reaching of any number $n$ is a Hydra Game for the Golden Automaton, we have that the Golden Automaton cannot fail to finitely reach any natural number. 

\subsection{Can Peano Arithmetic also prove the Collatz conjecture?}
If it is now sure that any system strong enough to prove the convergence of Goodstein's series also proves the Collatz conjecture, it could very well be possible to prove it from Peano arithmetic alone. In this final subsection, we intend to outline a strategy towards such a demonstration by defining a different game than the Hydra one and in particular, a zero-player game that is significantly simpler than John Conway's Game of Life and played on the complete binary tree $\{2\mathbb{N}+1;G,S\}$. 

In this cellular automaton, each cell is identified by a unique odd number and can only adopt three states: 
\begin{itemize}
\item \textbf{Black}, meaning the odd number is not (yet) proven to converge under the iterated Collatz transformation or equivalently that it is only equivalent to another black number
\item \textbf{Gold}, meaning the odd number is proven to converge and the consequences of its convergence have not yet been computed, \textit{ie.} it can have an \textit{offspring}
\item \textbf{Blue}, meaning the number is proven to converge and the consequences of its convergence have been computed \textit{ie.} its \textit{offspring} has already been turned gold
\end{itemize}

In this \textit{ad hoc} yet simpler game of life each gold cell yields and \textit{offspring} then turns blue, and whenever a cell is blue or gold its odd number is proven to converge. Starting with one cell colored in gold at the positions $1$, it applies the following algorithm to each gold cell in the natural order of odd numbers: 
\begin{enumerate}
    \item \textbf{Rule 1}: if a cell on $x$ is gold color cell on $V(x)$ gold
    \item \textbf{Rule 2}: if a cell on $x$ is gold, color cell on $S(x)$ gold depending on the precise conditions of rule 2
    \item If a cell on $a$ of type A is gold, then color that on $R_a(x)$ in gold
    \item If a cell on $c$ of type C is gold, then color that on $R_c(x)$ in gold
    \item After applying the previous rules on a gold cell, turn it blue
    \vspace{5pt}
    \\ Note that applying $R_b$ on a type B number being equivalent to \textbf{Rule 1} then $R_c$ the algorithm needs not implement a defined $R_b$
\end{enumerate}

Whenever a complete series of odd numbers between $2^n+1$ and $2^{n+1}-1$ has been colored in gold, it ticks it and returns what we will call its computational "expense", namely all the numbers colored blue and gold that are higher than $2^{n+1}-1$, thus giving a clear measurement of the algorithmic time it takes the Golden Automaton to prove the convergence of each complete level of the binary tree over $2\mathbb{N}+1$. We then plot the evolution of this expense on a linear and a logarithmic scale. 

\begin{figure}[H]
\minipage{0.32\textwidth}
  \includegraphics[width=\linewidth, height = 2.5cm]{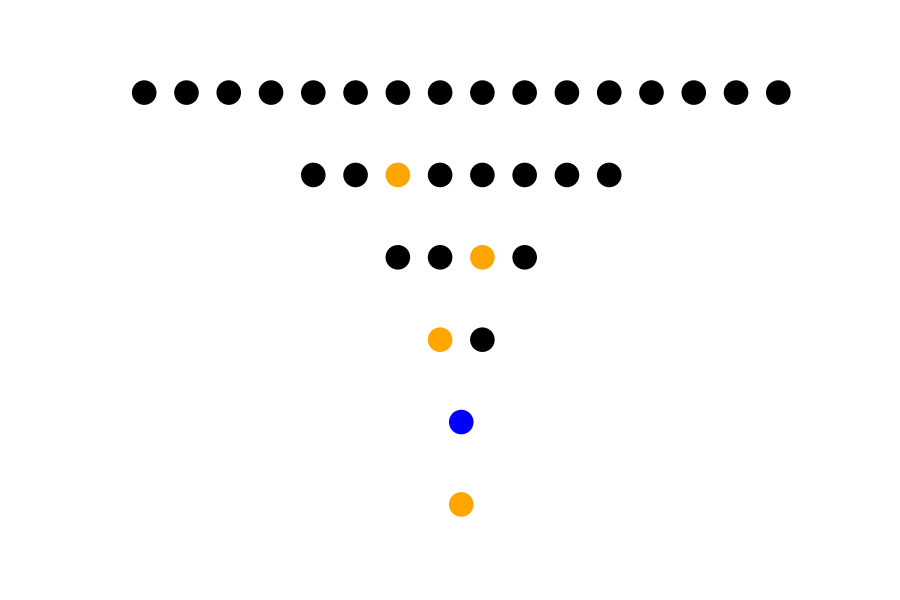}
\endminipage\hfill
\minipage{0.32\textwidth}
  \includegraphics[width=\linewidth, height = 2.5cm]{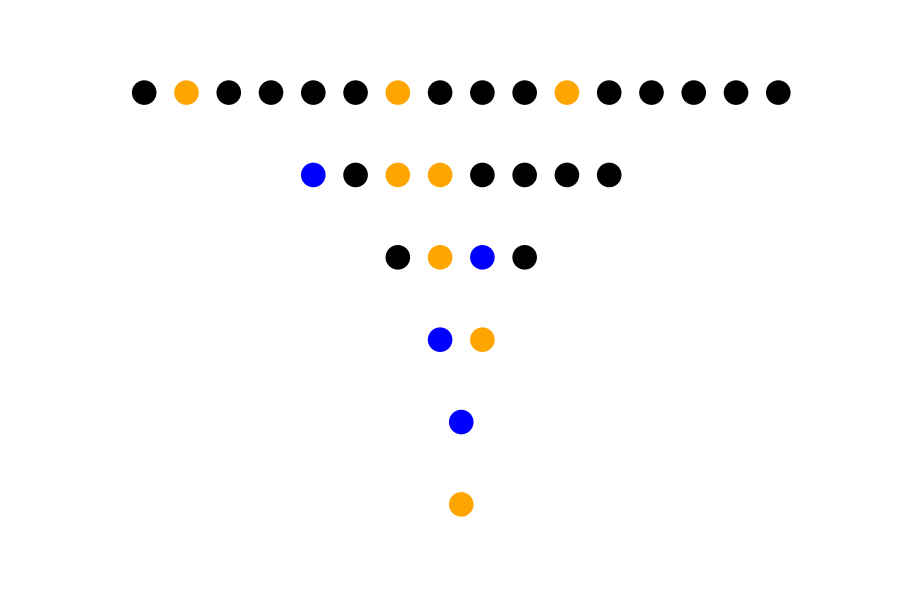}
\endminipage\hfill
\minipage{0.32\textwidth}%
  \includegraphics[width=\linewidth, height = 2.5cm]{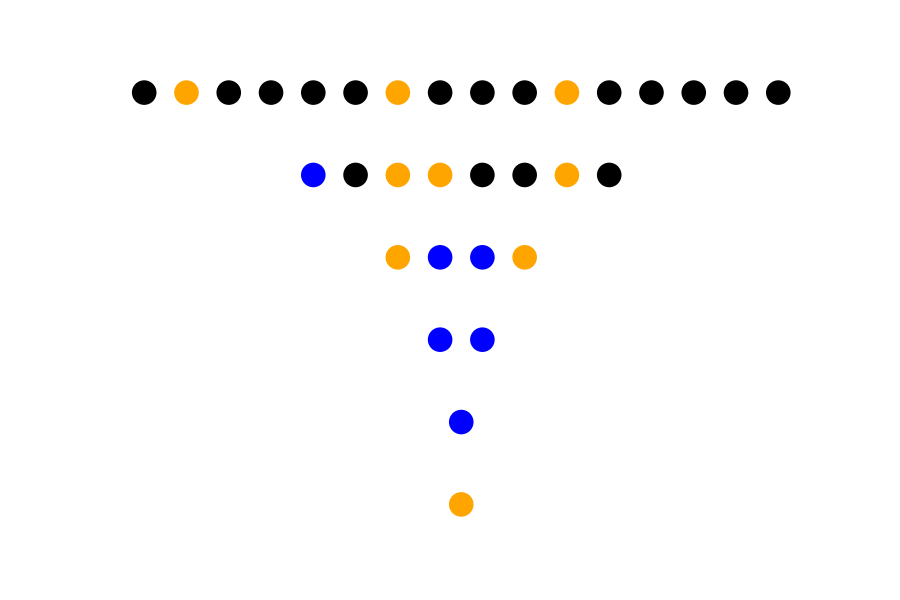}
\endminipage
\caption{case n=6, illustrating the principles of the game we defined. On the middle image, row \{5;7\} has been solved , with an "expense" of 8 numbers also solved above it. On the right image, row \{9; 11; 13; 15\} with an expense of 6. As number $1$ is the neutral element of operation $R_c$ we leave it in gold during all the simulation }\label{fig:dot_plot_6_0}
\end{figure}

\begin{figure}[H]
\includegraphics[clip, trim=19.5cm 4cm 15cm 2cm, width=14cm,height=6cm]{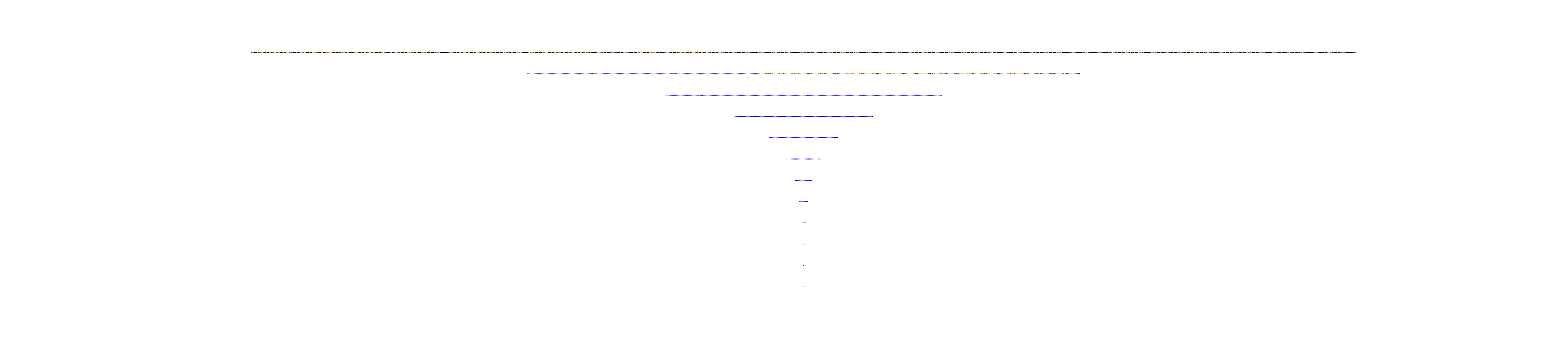}
\caption{case 12, seventh row completed}
\end{figure}

\begin{figure}[H]
\includegraphics[clip, trim=19.5cm 4cm 15cm 2cm, width=14cm,height=6cm]{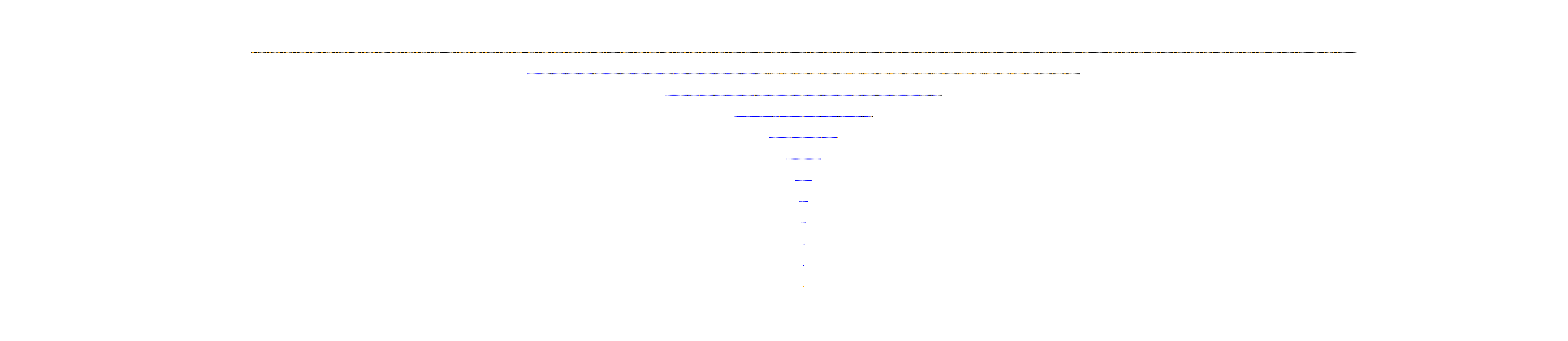}
\caption{case 12, eighth row completed}
\end{figure}

\newpage
To facilitate the observation of each row of the binary tree being covered by the Golden Automaton we here zoom into each of them individually:
\begin{figure}[H]
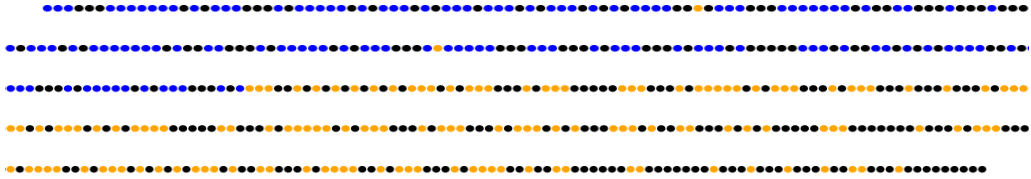

\includegraphics[clip, trim=39cm 19.7cm 70cm 5.3cm, height = 0.5cm ,width=1\textwidth]{figures/12_6.png}
\includegraphics[clip, trim=46.85cm 19.7cm 62.1cm 5.3cm, height = 0.5cm, width=1\textwidth]{figures/12_6.png}
\includegraphics[clip, trim=54.73cm 19.7cm 53.5cm 5.3cm, height = 0.5cm, width=1\textwidth]{figures/12_6.png}
\includegraphics[clip, trim=63.34cm 19.7cm 44.9cm 5.3cm, height = 0.5cm, width=1\textwidth]{figures/12_6.png}
\includegraphics[clip, trim=71.95cm 19.7cm 36cm 5.3cm, height = 0.5cm, width=1\textwidth]{figures/12_6.png}
\caption{zoom of row 11 (going from 1025 to 2047: each line has about 100 dots)}
\label{fig:zoom_row11}
\end{figure}
\vspace{-0.5em}
\begin{figure}[H]
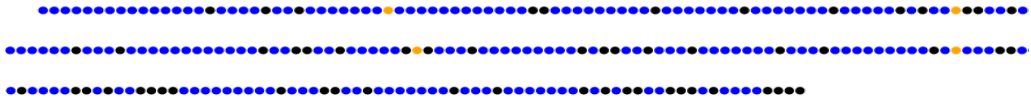

\includegraphics[clip, trim=49.35cm 18.15cm 60.1cm 6.9cm, height = 0.5cm, width=1\textwidth]{figures/12_6.png}
\includegraphics[clip, trim=56.75cm 18.15cm 52.61cm 6.9cm, height = 0.5cm, width=1\textwidth]{figures/12_6.png}
\includegraphics[clip, trim=64.23cm 18.15cm 45cm 6.9cm, height = 0.5cm, width=1\textwidth]{figures/12_6.png}
\caption{Row 10 (513 to 1023)}
\label{fig:zoom_row10}
\end{figure}
\vspace{-0.5em}
\begin{figure}[H]
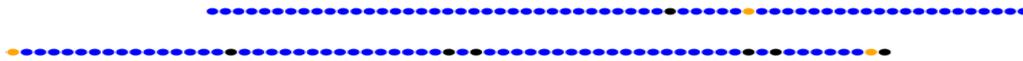

\includegraphics[clip, trim=53.5cm 16.55cm 57.05cm 8.45cm, height = 0.5cm, width=1\textwidth]{figures/12_6.png}
\includegraphics[clip, trim=59.8cm 16.55cm 51cm 8.45cm, height = 0.5cm, width=1\textwidth]{figures/12_6.png}
\caption{zoom Row 9 (257 to 511)}
\label{fig:zoom_row9}
\end{figure}
\vspace{-0.5em}
\begin{figure}[H]
\includegraphics[clip, trim=57cm 15cm 53.8cm 10.1cm, height = 0.5cm, width=1\textwidth]{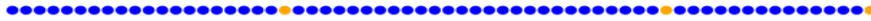}
\caption{Row 8 (129 to 255)}
\label{fig:zoom_row8}
\end{figure}

\begin{figure}[H]
\minipage{0.5\textwidth}
\includegraphics[clip, trim=0cm 0cm 0cm 0cm, height=120pt, width = 6cm]{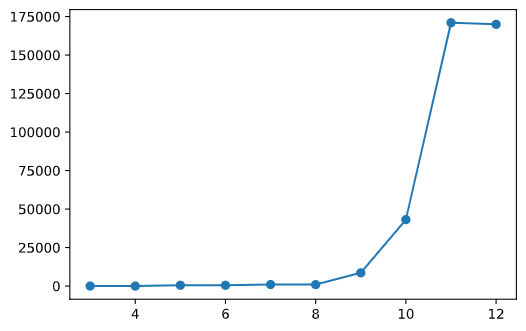}
\endminipage
\hspace{10px}
\minipage{0.2\textwidth}
\includegraphics[clip, trim=0cm 0cm 0cm 0cm, height=120pt, width=6cm]{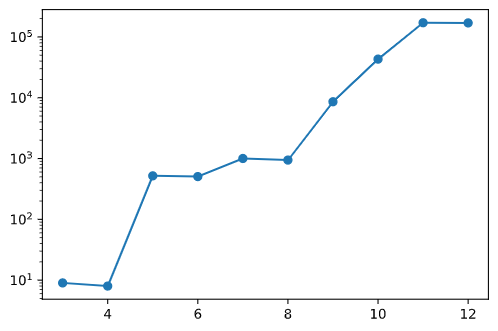}
\endminipage
\caption{Charting expenses to proof the convergence of each row of the binary tree over $2\mathbb{N}+1$. For example, the amount of gold and blue dots above the row going from 2049 to 4095 (row 12) is slightly below 175000. The total expense emerges from the generated tree's height. The same plot against a logarithmic scale (Right) indicates a line-like shape.}
\label{fig:expenses}
\end{figure}

From there we can thus provide two strategies to finalise a proof of the Collatz conjecture without the axiom of choice (which is needed to demonstrate no Hydra Game can be lost) and precisely within Peano arithmetic. The first strategy would consist of using automated theorem proving to single out the linear behavior we expose in \textbf{Figure 14} as a provable property of our game. The second - and we believe the most promising - would consist of analysing the average reproductive rate of gold dots, and demonstrating at any level $n$ of the binary tree they cannot fail to finitely take over any population of black dots below it.

\begin{definition}(Reproductive Rate)
\begin{enumerate}
    \item The \textbf{reproductive rate} of any golden dot on coordinate $n$ is the number of golden and blue dots it generates that are at or below coordinate $V(n)$. 
    \item The average reproductive rate of all black dots converges to 3,5 new black dots generated from $x$ under or equal to $V(x)$ . Indeed for any $x$, in the time to reach $V(x)$ only the offspring below $G(x)$ gets to generate new black dots under the rules of the binary tree, $S(x)$ can only reproduce once by applying $G(S(x))$, thus generating $V(x)$, and all numbers between $S(x)$ and the next Mersenne number cannot reproduce. More precisely, for any odd number $x$, there are $x+G^{-1}(x)$ odd numbers between itself and $V(x)$ included.
\end{enumerate}
\end{definition}

We can now count the average or limit reproductive rate of each type of golden dot. 
\begin{itemize}
\item for type B numbers we always have V(b) in the offspring, and also S(b) one out of two times, so the average reproductive rate is 1,5
\item for type C numbers we always have V(c) and R(c) in the offspring, and S(c) one out of two times so the average reproductive rate is 2,5
\item for type A numbers we have a more interesting converging series defining the average number of gold dots generated below V(a)+2. To easily count the offspring we use the property of the golden Automaton that it can only output $A_g$ through operations $R_b$ or $R_c$, and can only output non-Aup through operations V and S.  
    
The net offspring below $V(A_g)+2$ of $A_g$ numbers is defined by a branching process:
\begin{itemize}
\item $R_a(A_g)$ outputs either an Aup (in proportion of $\frac{1}{2}$) or equally either a Bup or a Cup ($\frac{1}{4}$ each)
\item then again $R_a(Aup)$ outputs either an Aup ($\frac{1}{2}$) or equally either a Bup or a Cup (again, $\frac{1}{4}$ each)
\item how long $R_a(A_g)$ keeps rendering an Aup only depends on $n$ where $A_g=G(3^n x)$ with $x$ non B (per \textbf{Rule 5})
\end{itemize}
        
So the formula of the average offspring below V(a)+2 of all type $A_g$ number, which can be written $G(3^n x)$ with x non B is $3n+2,5$. Since one out of two $A_g$ can be written $3^1 x$, a quarter can be written $3^2 x$ and so on we have the general formula that $A_g$ numbers up to $G(3^n)$ would have an average reproductive rate, which limit we can now determine:
\[
\lim_{n\to\infty} \sum_{i=1}^{n} \frac{1}{2^{i}}\cdot(3{i}+2,5)=6,5
\]
And this formula does not even account for the accumulated offspring of all the C and B numbers also colored gold in the process, and still below $V(A_g)$ so the average net reproductive rate of $A_g$ numbers is converging to a strictly greater value than this limit. 
\end{itemize}

As we know that the reproductive rate of black dots below $V(x)+2$ converges to $3,5$, when the average reproductive rate of all A, B and C type numbers generated by the Golden Automaton grows beyond $3,5$ we can be certain it can always finitely finish any row. As we already have $\frac{1,5+2,5+6,5}{3}=3,5$ since we did not count the offspring of C and B type numbers solved by $A_g$ numbers  and still below $V(A_g)$ in the computing of their birthrate, we can prove the average birthrate of golden numbers tends to equal $3,5 + \epsilon$ with $\epsilon > 0$, which finishes the Peano-arithmetical proof.

\section{Dedication, Attribution and Acknowledgements}
This work was supported by a personal grant to I. Aberkane from Mohammed VI University, Morocco and by a collaboration between Capgemini, Potsdam University, The Georg Simon Ohm University of Applied Sciences, and Strasbourg University. I. Aberkane created the framework of studying the Collatz dynamic in the coordinate system defined by the intersection of the binary and ternary trees over $2\mathbb{N}+1$, identified and demonstrated the five rules and predicted they would be isomorphic to a Hydra game over the set of undecided Collatz numbers, which he defined as well, allowing for a final demonstration of the Collatz Conjecture; he also outlined and computed the strategy of using reproductive rates of dots to define a Peano-arithmetical proof. Contributing equally, E. Sultanow and A. Rahn designed and coded an optimised, highly scalable graphical implementation of the five rules and ran all the simulations, confirming the Hydra game isomorphism and computing the first ever dot plot of the Golden Automaton over odd numbers, which they optimised as well. They were also the first team to ever simulate the five rules to the level achieved in this article, and to confirm their emerging geometric properties on such a scale, including the linearity of their logarithmic scaling and the limit reproductive rates of single dots of the golden automaton. 

I. Aberkane wishes to thank the late Prof. Solomon Feferman and Prof. Alan T. Waterman Jr, along with Prof. Paul Bourgine, Prof. Yves Burnod, Prof. Pierre Collet, Dr. Françoise Cerquetti and Dr. Oleksandra Desiateryk. 

The authors dedicate this work to the memory of John Horton Conway (1937-2020), Solomon Feferman (1928-2016) and Alan T. Waterman Jr (1918-2008). 

\vspace{1em}
\bibliographystyle{unsrt}
\bibliography{main} 

\end{document}